\documentclass{amsart}
\usepackage[foot]{amsaddr}
\usepackage{amssymb,amsfonts,amsmath,amsthm,amscd,cite,stmaryrd,dsfont,esint,graphicx,hyperref,standalone,tikz,upgreek,xcolor}
\usepackage[title]{appendix}
\usepackage[mathcal]{euscript}

\usepackage{fullpage}
\usetikzlibrary{positioning}

\theoremstyle{plain}
\newtheorem{thm}{Theorem}
\newtheorem{cor}{Corollary}
\newtheorem{lem}[cor]{Lemma}
\newtheorem{prop}[cor]{Proposition}

\theoremstyle{definition}
\newtheorem{defn}[cor]{Definition}

{}
\let\P\undefined{}

\DeclareMathOperator{\rot}{rot}

\DeclareMathOperator*{\argmin}{arg\,min}
\DeclareMathOperator{\Homeo}{Homeo}

\DeclareMathOperator{\Leb}{Leb}

\DeclareMathOperator{\radius}{radius}

\DeclareMathOperator{\dist}{dist}

\DeclareMathOperator{\SL}{SL}
\DeclareMathOperator{\mix}{mix}

\newcommand{\R}{\mathbb{R}}
\newcommand{\Z}{\mathbb{Z}}
\newcommand{\N}{\mathbb{N}}
\newcommand{\Q}{\mathbb{Q}}
\newcommand{\E}{\mathbb{E}}
\newcommand{\P}{\mathbb{P}}
\renewcommand{\tilde}{\widetilde}
\renewcommand{\hat}{\widehat}
\numberwithin{equation}{section}

\title{Exponential mixing by shear flows}
\author{William Cooperman}
\address{University of Chicago, Department of Mathematics. Chicago, Illinois}
\email{billc@uchicago.edu}
\begin{document}
\begin{abstract}
    We prove a version of Bressan's mixing conjecture where the advecting field is constrained to be a shear at each time. Also, inspired by recent work of Blumenthal\textendash{}Coti~Zelati\textendash{}Gvalani, we construct a particularly simple example of a shear flow which mixes at the optimal rate. The constructed vector field alternates randomly in time between just two distinct shears.
\end{abstract}
\keywords{mixing, transport, shear}
\subjclass{35Q49 (Primary) 37H05 (Secondary)}
\maketitle

\section{Introduction}
Given a divergence-free vector field $b \colon \R^2 / {(2\pi\Z)}^2 \to \R^2$ on the torus, we are interested in how effectively some mean-zero initial data $u_0$ is mixed when advected by $b$. By solving the transport equation
    \begin{equation}\label{eq:transport}
        \begin{cases}
            D_t u(t, x) + b(t, x) \cdot D_x u(t, x) = 0 &\qquad \text{ for $t > 0$ and $x \in \R^2 / {(2\pi\Z)}^2$}\\
            u(0, x) = u_0(x) &\qquad \text{ for $x \in \R^2 / {(2\pi\Z)}^2$.}
        \end{cases}
    \end{equation}
we can measure mixing by taking various measurements of $u(1, \cdot)$.

There are a few natural ways to measure mixing, two of which we list here. The functional mixing scale measures $\|u(1, \cdot)\|_{H^{-1}}$ or some other negative Sobolev norm. There are several works studying how bounds on the functional mixing scale depend on various norms of the advecting field, e.g.\ the energy, palenstrophy, etc.~\cite{OptimalStirring, Zillinger, CrippaLucaSchulze, IyerKiselevXu, ThiffeaultMultiscale, MathewMezicPetzold, ZillingerComparison, Anomalous} On the other hand, the (related, but not equivalent) geometric mixing scale measures the size of the largest ball $B$ on which the mean $\left|\frac{1}{|B|} \int_B u(1, x) \; \mathrm{d}x\right|$ is greater than $1$ (or any other fixed constant). Bressan~\cite{Bressan} conjectured that the geometric mixing scale is bounded below by $\exp(-C\|D_x b\|_{L^1})$ for some constant $C > 0$. By proving new estimates for the regular Lagrangian flow, Crippa\textendash{}De~Lellis~\cite{Crippa-de-Lellis} bounded the geometric mixing scale by $\exp(-C\|D_x b\|_{L^p})$ for all $p > 1$, where the constant $C > 0$ depends on $p$.

In this paper, we study the $p = 1$ case under the assumption that $b$ is a shear, that is, that at each time $t$, all the vectors $b(t, \cdot)$ are parallel. Since, for example, piecewise constant shears do not lie in $W^{1,p}$ for $p > 1$ unless they are constant everywhere, the transition from $p > 1$ to $p = 1$ leads to a qualitative change in the allowable vector fields $b$. Seeger\textendash{}Smart\textendash{}Street~\cite{Multilinear} proved general harmonic analysis estimates inspired by Bianchini's approach~\cite{Bianchini} to solve Bressan's conjecture in the one-dimensional case, and Had\v{z}i\'{c}\textendash{}Seeger\textendash{}Smart\textendash{}Street~\cite{HSSS} found applications of these estimates to bound the mixing scale. In~\cite{HSSS}, an example involving shears is given which presents a clear obstruction to extending Bianchini's approach to dimensions higher than one.
To better understand the difficulty, Had\v{z}i\'{c}\textendash{}Seeger\textendash{}Smart\textendash{}Street~\cite{HSSS} posed a simplified discrete version of Bressan's conjecture which allows only shears. We solve this simplified version with Theorem~\ref{thm:main-bound}.

First, we define the geometric mixing scale in Bressan's sense.

\begin{defn}
    Given a vector field $b \in L^1([0, 1]; W^{1,1}(\R^2 / {(2\pi\Z)}^2; \R^2))$, we define the mixing scale $\mix(b)$ by \[ \mix(b) := \sup \left\{\radius(B) \mid \left|\frac{1}{|B|}\int_B u(1, x) \; \mathrm{d}x\right| > 1 \right\}, \] where the supremum is over all balls $B \subseteq \R^2 / {(2\pi\Z)}^2$ and $u(t, x)$ is the solution to the transport equation~\eqref{eq:transport} with initial data $u_0 = 2\mathds{1}_{x_1 \leq \pi} - 2\mathds{1}_{x_1 > \pi}$.
\end{defn}

\begin{thm}\label{thm:main-bound}
    Let $b \in L^1([0, 1]; W^{1,1}(\R^2 / {(2\pi\Z)}^2; \R^2))$ be a divergence-free shear at every time $t \in [0, 1]$, that is, we assume that $b(t, x)$ is parallel to $b(t, y)$ for all $t \geq 0$ and $x, y \in \R^2 / {(2\pi\Z)}^2$. Then there is a constant $C > 0$ such that \[ |\log \mix(b)| \leq C\|D_x b\|_{L^1}. \]
\end{thm}

Our second result is an example of a particularly simple shear flow which mixes at the optimal rate. The argument follows the path of the recent work of Blumenthal\textendash{}Coti~Zelati\textendash{}Gvalani~\cite{BCZG}. To explain, we define the horizontal sine field \[ b_{\text{horiz}}(x) := (\sin(x_2), 0) \] and the vertical sine field \[ b_{\text{vert}}(x) := (0, \sin(x_1)). \]
In~\cite{BCZG}, Blumenthal\textendash{}Coti~Zelati\textendash{}Gvalani establish a general framework to show that the vector field \[ b^{\bar{\omega}}(t, x) := \begin{cases} b_{\text{horiz}}(x_1, x_2+\omega_{\lceil t \rceil}) &\qquad \text{ if $\lceil t \rceil$ is odd}\\ b_{\text{vert}}(x_1+\omega_{\lceil t \rceil}, x_2) &\qquad \text{ if $\lceil t \rceil$ is even}\end{cases} \] almost surely mixes at an exponential rate, where the numbers $\bar{\omega} := (\omega_1, \omega_2, \dots)$ are chosen independently and uniformly at random from $[0, 2\pi]$, which proves a conjecture of Pierrehumbert~\cite{Pierrehumbert}.

    In this paper, we define instead the vector field \[ b^{\bar{\tau}}(t, x) := \begin{cases} \tau_{\lceil t \rceil}b_{\text{horiz}}(x) &\qquad \text{ if $\lceil t \rceil$ is odd}\\ \tau_{\lceil t \rceil}b_{\text{vert}}(x) &\qquad \text{ if $\lceil t \rceil$ is even}\end{cases} \] where the numbers $\bar{\tau} = (\tau_1, \tau_2, \dots)$ are chosen independently and uniformly at random from $[0, T]$ for some large $T > 0$. In other words, we choose to randomize the duration for which each sine field runs, instead of randomizing the phase (for simplicity we choose to randomize the magnitude instead, but by rescaling time this is equivalent to randomizing duration). This modification creates some complications; namely, there are fixed points under the flow of $b^{\bar{\tau}}$. Indeed, $b^{\bar{\tau}}(x) = 0$ for any $x \in {\{0, \pi\}}^2$. To deal with this, we consider mixing on the space $(\R^2 / {(2\pi\Z)}^2) \setminus {\{0, \pi\}}^2$, and introduce some new arguments to handle the lack of compactness caused by removing the fixed points.

\begin{figure}[htb]
    \centering
    \includegraphics{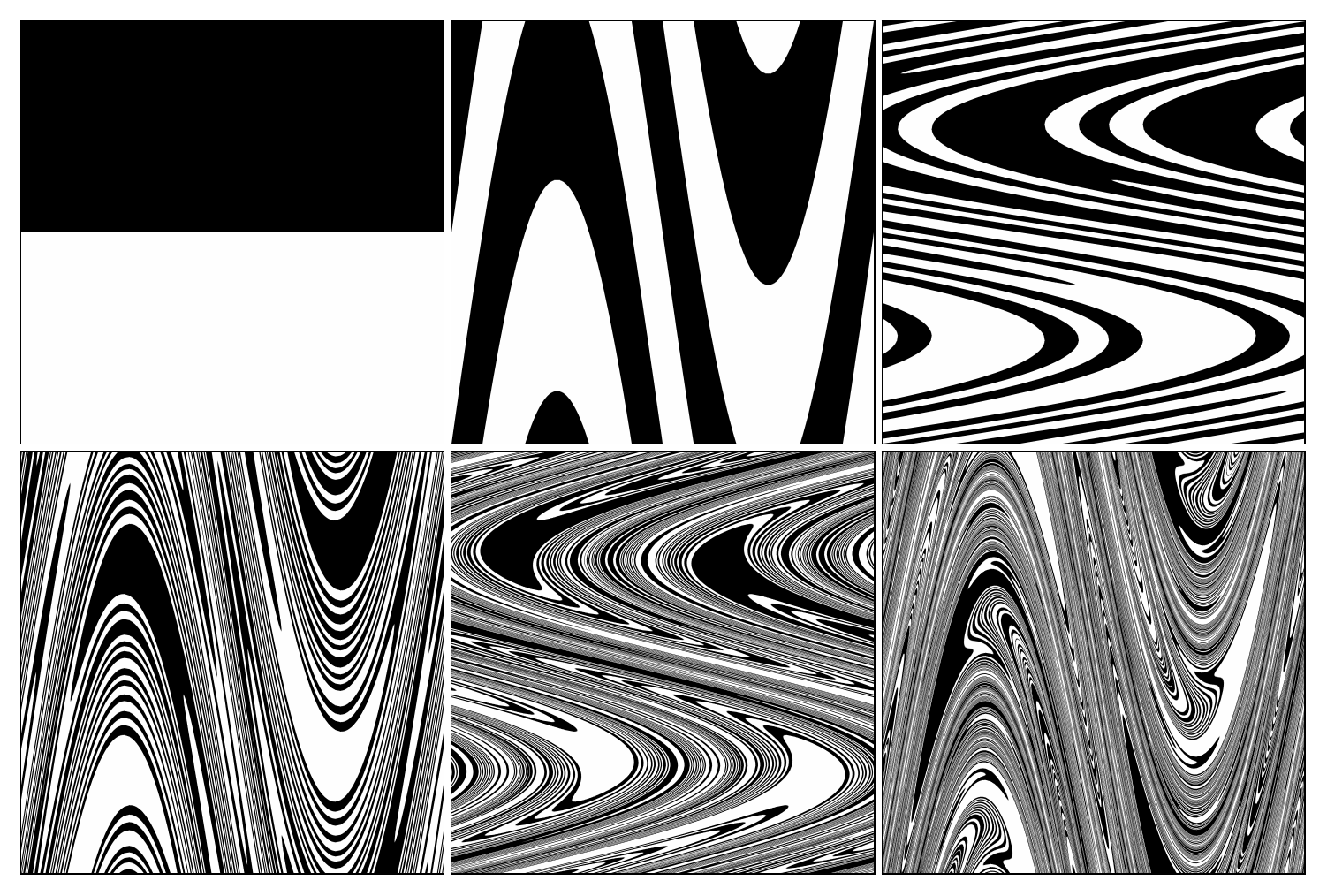}
    \caption{The flow with random durations}
\end{figure}

To make our notation match with the mixing scale, we write $b^{\bar{\tau}}_\alpha(t, x) := \alpha b^{\bar{\tau}}(\alpha t, x)$ to denote speeding time up by a factor of $\alpha$, and we write $u^{\bar{\tau}}$ to denote the solution to~\eqref{eq:transport} with $b = b^{\bar{\tau}}$.
\begin{thm}\label{thm:main-example}
    If $T > 0$ is sufficiently large and $\bar{\tau} = (\tau_1, \tau_2, \dots)$ is a random sequence chosen independently and uniformly from $[0, T]$, then there is a random variable $c = c(T) > 0$, which is positive almost surely, such that \[ |\log \mix(b^{\bar{\tau}}_\alpha)| \geq c \|D_x b^{\bar{\tau}}_\alpha\|_{L^1} \] for any $\alpha > 0$ almost surely.

    More generally, for any initial data $u_0 \in L^\infty(\R^2 / {(2\pi\Z)}^2)$ with mean zero, if $u$ solves~\eqref{eq:transport} then there is a random variable $C = C(u_0, T) > 0$ which is finite almost surely and satisfies \[ \sup \left\{ t > 0 \mid \frac{1}{|B|}\int_B u^{\bar{\tau}}(t, x) \; \mathrm{d}x > 1 \right\} \leq C_0(\bar{\tau})\log(|B|) \] for all balls $B \subseteq \R^2 / {(2\pi\Z)}^2$ and almost every $0 \leq \bar{\tau} \leq T$.
\end{thm}

Finally, we show that even if $T > 0$ is arbitrarily large, it is possible that $b^{\bar{\tau}}$ mixes at a subexponential rate if $\tau_n = T$ for all $n \in \N$, which answers a question of Blumenthal\textendash{}Coti~Zelati\textendash{}Gvalani (see Remark 1.2 of~\cite{BCZG}) in the negative.

\section{Shear flows can mix no faster than exponentially}
%
%
In this section, we prove Theorem~\ref{thm:main-bound}. Our argument is inspired by Crippa\textendash{}De~Lellis~\cite{Crippa-de-Lellis} and De~Lellis~\cite{De-Lellis}. First, define the flow induced by $b$ to be the function $\Phi \colon [0, 1] \times (\R^2 / {(2\pi\Z)}^2) \to \R^2 / {(2\pi\Z)}^2$, where $\Phi(\cdot, x)$ solves the ordinary differential equation \[ \begin{cases} D_t \Phi(t, x) = b(\Phi(t, x)) &\quad \text{for $t > 0$ and $x \in \R^2 / {(2\pi\Z)}^2$}\\ \Phi(0, x) = x &\quad \text{for $x \in \R^2 / {(2\pi\Z)}^2$}. \end{cases} \] We define a suitable notion of energy for the flow $\Phi$, and then show that it cannot grow too quickly compared to the bound on the advecting field $b$. In De~Lellis~\cite{De-Lellis}, the energy is chosen to be \[ \int_U \left|\log |\Phi(t, x)-\Phi(t, y)|\right| \; \mathrm{d}(x, y), \] where $U = \{(x, y) | x_1 < \pi \text{ and } y_1 \in (5\pi/4, 7\pi/4) \text{ and } |\Phi(1, x)-\Phi(1, y)| < \varepsilon\}$, which works for the $p > 1$ case. In our setting, we define the energy $E$ below to work for the $p=1$ case when $b$ is a shear. Instead of tracking distances between a fixed set of $(x, y)$ pairs, we track the distance from each $x$ to the closest $y$ which started on the ``other side of the torus'' (so $x_1 < \pi \leq y_1$).

We now turn to the proof. By a standard density argument, we assume without loss of generality that $b \in C^\infty([0, 1] \times (\R^2 / {(2\pi\Z)}^2); \R^2)$. Notationally, we write $C > 0$ to denote a constant which may change from line to line, and write $[a, b] = [b, a]$ interchangeably; both refer to the interval $[\min\{a, b\}, \max\{a, b\}]$.

Fix $\varepsilon > 0$ and suppose $\mix(b) \leq \varepsilon$. We define the energy $E \colon [0, 1] \to \R$ by \[ E(t) := \int_L |\log \dist(\Phi(t, x), \partial\Phi(t, L))| \; \mathrm{d}x, \] where $L := \{x \in \R^2 / {(2\pi\Z)}^2 \mid 0 < x_1 < \pi\}$. It's clear that $E(0) \leq C$ and, since $\mix(b) \leq \varepsilon$, that $E(1) \geq C^{-1} |\log \varepsilon|$ for some constant $C > 0$.

To prove Theorem~\ref{thm:main-bound}, we will show that $E'(t) \leq C\|D_x b(t, \cdot)\|_{L^1}$ for almost every $0 \leq t \leq 1$.
For ease of notation, let $d(t, x) := \dist(\Phi(t, x), \partial\Phi(t, L))$. Fix $0 \leq t \leq 1$ and assume, without loss of generality, that $b(t, x) = (0, b_2(t, x_1))$ (otherwise, a rotated version of the same argument goes through). Let $f \colon \Phi(t, L) \to \partial \Phi(t, L)$ be measurable such that, for any $y \not\in \Phi(t, L)$, either $|y-x| > |f(x)-x|$ or \[ |y-x|=|f(x)-x| \text{ and } (b(t,y)-b(t,x))\cdot(y-x) \geq (b(t, f(x))-b(t, x))\cdot(f(x)-x). \] Write \[ \alpha(x) := \argmin_{t \in {f(x)}_1-x_1 + 2\pi\Z} |t|. \]  Now we compute
\begin{align*}
    E'(t) &\leq \int_0^{2\pi} \int_0^{2\pi} \mathds{1}_{\Phi(t, L)}(x)\left|\frac{D_t d(t, x)}{d(t, x)}\right| \; \mathrm{d}x_1 \; \mathrm{d}x_2\\
    &\leq \int_0^{2\pi} \int_0^{2\pi} \mathds{1}_{\Phi(t, L)}(x)\left|\frac{(b(t, f(x))-b(t, x))\cdot(f(x)-x)}{{d(t, x)}^2}\right| \; \mathrm{d}x_1 \; \mathrm{d}x_2\\
    &\leq \int_0^{2\pi} \int_0^{2\pi} \mathds{1}_{\Phi(t, L)}(x)\left|\frac{(b_2(t, {f(x)}_1)-b_2(t, x_1))({f(x)}_2-x_2)}{{d(t, x)}^2}\right| \; \mathrm{d}x_1 \; \mathrm{d}x_2\\
    &\leq \int_0^{2\pi} \int_0^{2\pi} \mathds{1}_{\Phi(t, L)}(x) \left|\frac{{f(x)}_2-x_2}{{d(t, x)}^2}\right| \int_{x_1}^{x_1+\alpha(x)} D_{x_1} b_2(t, \gamma)\; \mathrm{d}\gamma \; \mathrm{d}x_1 \; \mathrm{d}x_2\\
    &\leq \int_0^{2\pi} \int_0^{2\pi} |D_x b_2(t, \gamma)| \int_0^{2\pi} \mathds{1}_{\Phi(t, L)}(x) \mathds{1}_{[x_1, x_1+\alpha(x)]}(\gamma) \left|\frac{{f(x)}_2-x_2}{{d(t, x)}^2}\right|\; \mathrm{d}x_1 \; \mathrm{d}\gamma \; \mathrm{d}x_2
\end{align*}

\begin{figure}
    \includegraphics[scale=0.5]{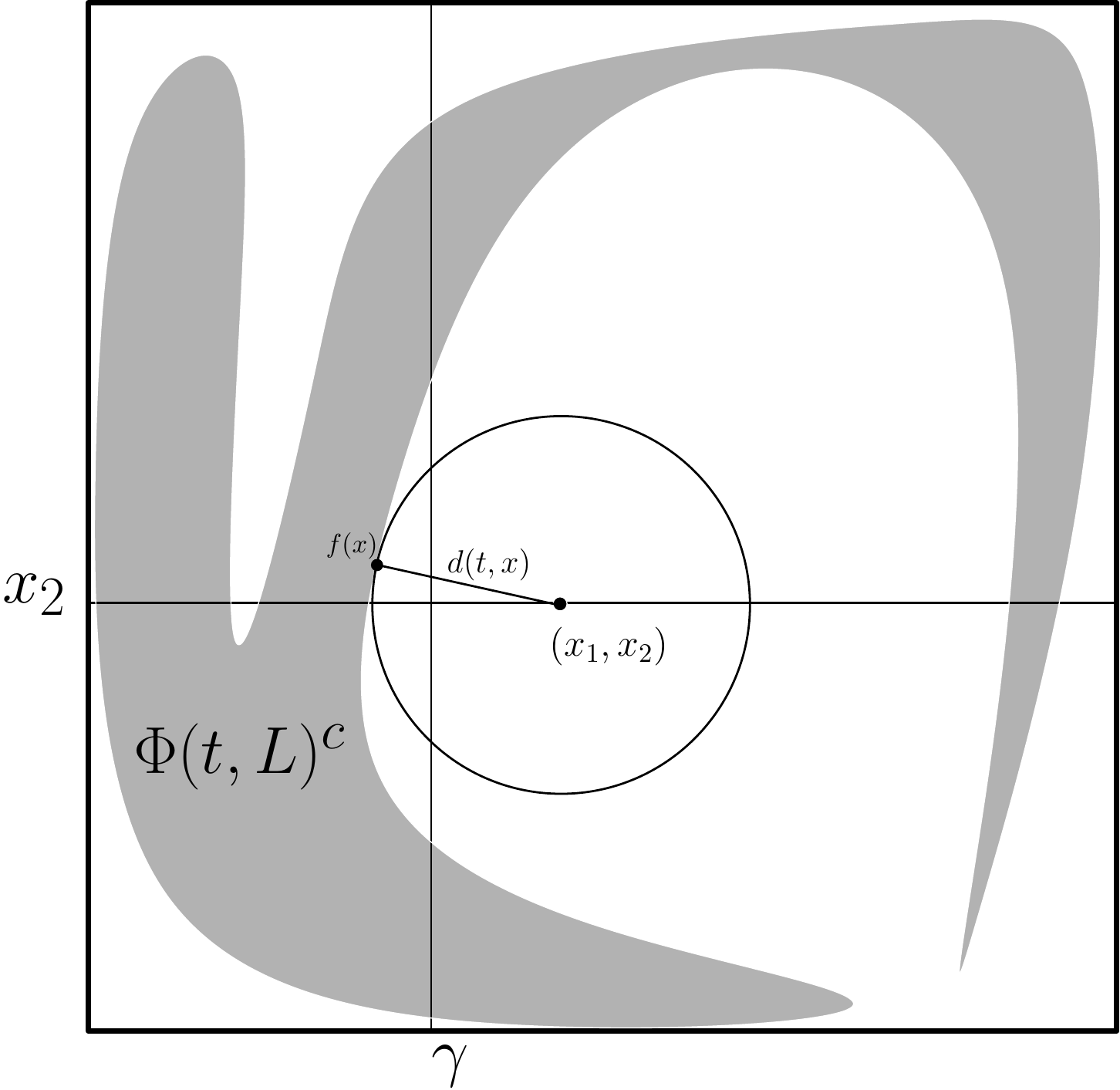}
    \caption{Our setup; the vector field $b$ is a vertical shear}
\end{figure}

From here, we note that \[ {d(t, x)}^2 = {({f(x)}_1-x_1)}^2 + {({f(x)}_2-x_2)}^2 = {\alpha(x)}^2 + {({f(x)}_2-x_2)}^2, \] so
\begin{align*}
    E'(t) &\leq \int_0^{2\pi} \int_0^{2\pi} |D_x b_2(t, \gamma)| \int_0^{2\pi} \mathds{1}_{\Phi(t, L)}(x) \mathds{1}_{[x_1, x_1+\alpha(x)]}(\gamma) \left|\frac{\sqrt{{d(t, x)}^2-{\alpha(x)}^2}}{{d(t, x)}^2}\right|\; \mathrm{d}x_1 \; \mathrm{d}\gamma \; \mathrm{d}x_2.\\
    &=: \int_0^{2\pi} \int_0^{2\pi} |D_x b_2(t, \gamma)| I(\gamma, x_2) \; \mathrm{d}\gamma \; \mathrm{d}x_2.
\end{align*}
It remains to show that $I(\gamma, x_2)$ is bounded by a constant for each $\gamma$ and $x_2$, so fix some $0 \leq \gamma, x_2 \leq 2\pi$ such that the integrand of $I(\gamma, x_2)$ is not identically zero. Note that the integrand is only nonzero when $(\gamma, x_2) \in B_{d(t, x)}(x) \subseteq \Phi(t, L)$, so $d(t, x) \geq d(t, (\gamma, x_2))/2 > 0$ whenever the integrand is nonzero (we take the convention that the integrand is zero whenever either indicator function is zero, regardless of the value of $d(t, x)$). In other words, whenever the integrand is nonzero, $d(t, x)$ is bounded away from zero. Let \[ d_\text{min} := \inf \{d(t, x) \mid x \in \Phi(t, L) \text{ and } \gamma \in [x_1, x_1 + \alpha(x)] \} \] and let $0 \leq y \leq 2\pi$ be such that $(y, x_2) \in \Phi(t, L)$ and $\gamma \in [y, y+\alpha(y, x_2)]$ and $d(t, (y, x_2)) \leq 2d_\text{min}$.

We break $I(\gamma, x_2)$ into two parts, depending on whether $|x_1-\gamma| \leq d(t, (y, x_2))$ or not:
\begin{align*}
    I(\gamma, x_2) &= \int_0^{2\pi} \mathds{1}_{[\gamma-d(t, (y, x_2)), \gamma+d(t, (y, x_2))]}(x_1)\mathds{1}_{\Phi(t, L)}(x) \mathds{1}_{[x_1, x_1+\alpha(x)]}(\gamma) \left|\frac{\sqrt{{d(t, x)}^2-{\alpha(x)}^2}}{{d(t, x)}^2}\right|\; \mathrm{d}x_1\\
    &\quad + \int_0^{2\pi} (1-\mathds{1}_{[\gamma-d(t, (y, x_2)), \gamma+d(t, (y, x_2))]}(x_1))\mathds{1}_{\Phi(t, L)}(x) \mathds{1}_{[x_1, x_1+\alpha(x)]}(\gamma) \left|\frac{\sqrt{{d(t, x)}^2-{\alpha(x)}^2}}{{d(t, x)}^2}\right|\; \mathrm{d}x_1\\
    &=: I_1(\gamma, x_2) + I_2(\gamma, x_2).
\end{align*}
To bound $I_1(\gamma, x_2)$, we note that the integrand is bounded above by $\frac{C}{d_\text{min}}$, and is only nonzero on a set of measure at most $4d_\text{min}$, so $I_1(\gamma, x_2) \leq C$.

On the other hand, if the integrand of $I_2(\gamma, x_2)$ is nonzero, then $|\alpha(x)| \geq |\gamma-x_1|$. Also, the triangle inequality yields \[ d(t, x) \leq |x_1-y| + d(t, (y, x_2)) \leq (|x_1-\gamma| + |y-\gamma|) + 2d_\text{min} \leq |x_1-\gamma| + 4d_\text{min}. \] Intuitively, the flow direction of $f(x)$ relative to $x$ is almost orthogonal to $f(x)-x$. This will imply that $d(t, x)$ cannot be increasing too quickly. Indeed, combining these inequalities yields \[ \frac{\sqrt{{d(t, x)}^2-{\alpha(x)}^2}}{{d(t, x)}^2} \leq \frac{\sqrt{{(|x_1-\gamma| + 4d_\text{min})}^2-{|x_1-\gamma|}^2}}{{|x_1-\gamma|}^2} \leq \frac{\sqrt{24d_\text{min}}}{{|x_1-\gamma|}^{3/2}}. \]
We use this bound to estimate \[ I_2(\gamma, x_2) \leq 2\int_{d_\text{min}}^{2\pi} \frac{\sqrt{24d_\text{min}}}{{|x_1-\gamma|}^{3/2}} \; \mathrm{d}x_1 \leq 8\sqrt{6}, \] so $I_2(\gamma, x_2)$ is bounded by a constant, as desired.

\section{Exponential mixing via random durations}
In this section, we prove Theorem~\ref{thm:main-example}. Rather than working with the vector fields $b_\text{horiz}$ and $b_\text{vert}$, we will instead work with the flows they generate. Define $f_1^\tau, f_2^\tau \in \Homeo(\R^2 / {(2 \pi \Z)}^2)$ by $f_1^\tau(x, y) := (x + \tau \sin(y), y)$ and $f_2^\tau(x, y) := (x, y + \tau \sin(x))$. Fix a large $T > 0$ and let the random variables $\bar{\tau} := (\tau_1, \tau_2, \dots)$ be a sequence of real numbers chosen independently and uniformly at random in $[0, T]$. We write $\P_n$ to denote the probability measure on $\bar{\tau}^n := (\tau_1, \dots, \tau_{2n})$, and $\P$ to denote the probability measure on $\bar{\tau}$, with the obvious coupling. Define the random map \[ \Phi^{\bar{\tau}}_n := (f_2^{\tau_{2n}} \circ f_1^{\tau_{2n-1}}) \circ (f_2^{\tau_{2n-2}} \circ f_1^{\tau_{2n-3}}) \circ \dots \circ (f_2^{\tau_2} \circ f_1^{\tau_1}). \] Sometimes, if we have only defined $\bar{\tau}^n$ but not all of $\bar{\tau}$, we will refer to the corresponding map as $\Phi^{\bar{\tau}^n}_n$ for emphasis. Note that the fixed points of $\Phi^{\bar{\tau}}_n$ are almost surely $F := {\{0, \pi\}}^2$. We regard $X := (\R^2 / {(2 \pi \Z)}^2) \setminus F$ as a random dynamical system, equipped with the $\Z$-action given by \[ n \cdot^{\bar{\tau}} x := \Phi^{\bar{\tau}}_n(x), \] which we refer to as the one-point chain.

\subsection{A sufficient condition for mixing}
Following the program of Blumenthal\textendash{}Coti~Zelati\textendash{}Gvalani~\cite{BCZG}, we also define the two-point and projective chains.
\begin{defn}
    The two-point chain is the space $X^2 \setminus \Delta$, where $\Delta = \{(x, x) \mid x \in X\}$ is the diagonal, equipped with the $\Z$-action given by \[ n \cdot^{\bar{\tau}} (x, y) \mapsto \tilde{\Phi}^{\bar{\tau}}_n(x, y) := (\Phi^{\bar{\tau}}_n(x), \Phi^{\bar{\tau}}_n(y)). \]
\end{defn}
\begin{defn}
    The projective chain is the unit tangent bundle $T^1 X$, equipped with the $\Z$-action given by \[ n \cdot^{\bar{\tau}} (x, v) \mapsto \hat{\Phi}^{\bar{\tau}}_n(x, v) := \left(\Phi^{\bar{\tau}}_n(x), \frac{\left(D_x\Phi^{\bar{\tau}}_n(x)\right) v}{\left|\left(D_x\Phi^{\bar{\tau}}_n(x)\right)v\right|}\right). \] If $\bar{\tau}$ is understood from context, we drop it from the notation.
\end{defn}
Note that, in the above definitions, the same random sequence $\bar{\tau}$ is used for both coordinates.

Next, we define the notion of geometric ergodicity. To show mixing in Bressan's sense, our main goal is to prove that the two-point chain has the following property.
\begin{defn}
    If $Y$ is a random dynamical system equipped with a $\Z$-action and $V \colon Y \to \R_{\geq 0}$, then we say that $Y$ is $V$-uniformly geometrically ergodic if $Y$ admits a unique stationary measure $\pi$, and there exist $C > 0$ and $\gamma \in (0, 1)$ such that, for any $y \in Y$ and $\varphi \colon Y \to \R$ measurable with $\varphi/(1+V) \in L^\infty(Y)$, \[ \left| \E[\varphi(n \cdot y)] - \int \varphi \mathrm{d}\pi \right| \leq CV(y)\|\varphi/(1+V)\|_{L^\infty} \gamma^n \] for all $n \geq 0$.
\end{defn}

In this paper, every unique stationary measure $\pi$ will be the usual Lebesgue measure (or, in the case of the projective chain, the Lebesgue measure times the uniform measure on $S^1$).

Now we show that uniform geometric ergodicity of the two-point chain implies mixing in Bressan's sense. The argument is inspired by that of Proposition~4.6 in~\cite{BCZG}.
\begin{prop}\label{prop:uge-implies-mixing}
    Suppose that the two-point chain is $V$-uniformly geometrically ergodic for some $V \in L^1(X^2 \setminus \Delta)$. Then for every $\varphi \in L^\infty(\R^2 / {(2\pi\Z)}^2)$ with $\int \varphi = 0$, there is a random variable $C_0 = C_0(V, \varphi, \bar{\tau}) > 0$, which is finite almost surely, such that for every ball $B \subseteq \R^2 / {(2 \pi \Z)}^2$, we have \[ \sup \left\{ n \in \N \mid \frac{1}{|B|}\int_B \varphi(n \cdot^{\bar{\tau}} x) \; \mathrm{d}x > 1 \right\} \leq C_0(\bar{\tau})\log(|B|).\]
\end{prop}
\begin{proof}
    Let $x, y \in X$. Then, by uniform geometric ergodicity, we have \[ \left|\E[\varphi(n \cdot x)\varphi(n \cdot y)]\right| \leq CV(x, y)\|\varphi\|_{L^\infty}\gamma^n. \] By Chebyshev's inequality, for any ball $B \subseteq \R^2 / {(2 \pi \Z)}^2$, we integrate the previous display to get \[ \P\left[\frac{1}{|B|}\int_B \varphi(n \cdot^{\bar{\tau}} x) \; \mathrm{d}x > \frac13\right] \leq C\|\varphi\|_{L^\infty}\gamma^n\frac{1}{{|B|}^2}\int_B \int_B V(x, y) \; \mathrm{d}x \; \mathrm{d}y. \] Covering $\R^2 / {(2 \pi \Z)}^2$ by at most $Cr^{-2}$ many balls of radius $\frac{r}{10}$, we conclude from the union bound that \[ \P\left[ \exists B \mid \radius(B) > r \text{ and } \frac{1}{|B|}\int_B \varphi(n \cdot^{\bar{\tau}} x) \; \mathrm{d}x > \frac12\right] \leq Cr^{-2}\|\varphi\|_{L^\infty}\gamma^n\frac{1}{r^2}\int_{X^2 \setminus \Delta} V(x, y) \; \mathrm{d}(x, y). \] Now set $n = k \log r$ for some large $k > 0$. Using the union bound over all $r = 1/m$ for $m \in \N$, we have \[ \P\left[\exists B \mid \exists n \geq k \log |B| \text{ and } \frac{1}{|B|}\int_B \varphi(n \cdot^{\bar{\tau}} x) \; \mathrm{d}x > 1\right] \leq C\sum_{m=1}^\infty m^{-k/2} \leq C{(k/2-1)}^{-1}, \] where we choose $k$ large enough for $r^{k/2}$ to cancel the polynomial terms above, and absorb the rest into the constant $C$.

    It follows that, setting $C_0$ to be the smallest constant which satisfies the claimed inequality, we have \[ \P[C_0 \leq k] \leq C{(k/2-1)}^{-1}, \] so $C_0$ is finite almost surely.
\end{proof}

Thus, to prove Theorem~\ref{thm:main-example}, it suffices to show that the two-point chain is $V$-uniformly geometrically ergodic for some integrable $V$.

\subsection{Uniform geometric ergodicity}
In the following, let $Y$ be one of the chains above (one-point, two-point, or projective) with the associated $\Z$-action.
\begin{defn}
    A set $E \subseteq Y$ is small if there exists a nontrivial measure $\mu$ on $Y$ and a natural number $n = n(E) \in \N$ such that \[ \P[n \cdot^\tau y \in F] \geq \mu(F) \] for every $y \in E$ and every measurable $F \subseteq Y$. If there exists such an $E$, we say that $Y$ admits an open small set.
\end{defn}
We now include a sufficient and easily-verifiable condition for the existence of a small set.
\begin{lem}[Blumenthal\textendash{}Coti~Zelati\textendash{}Gvalani~\cite{BCZG}, Proposition 3.1]\label{lem:small-set-condition}
    Suppose that there is a point $y \in Y$ and some $0 \leq \bar{\tau}^n_\star \leq T$ such that the map \[ \Psi_y \colon \bar{\tau}^n \mapsto \Phi^{\bar{\tau}^n}_n(y) \] is a submersion at $\bar{\tau}^n = \bar{\tau}^n_\star$. Then $Y$ admits a small set.
\end{lem}
\begin{proof}
    The hypotheses imply that $\bar{\tau}^n_\star$ lies in the support of $\P_n$. The proof is given in~\cite{BCZG}; the idea is that the constant rank theorem implies that $\Psi_y$ looks like (up to a change of charts) an orthogonal projection. The dilation of the pushforward ${(\Psi_y)}^*\P_n$ is controlled by the Jacobians of the charts, so we define $\mu = c\mathds{1}_U\Leb$, where $U$ is a small neighborhood of $\Psi_y(\bar{\tau}^n_\star)$ and $c > 0$ is sufficiently small compared to the $C^1$ norms of the charts.
\end{proof}
\begin{defn}
    A chain $Y$ is strongly aperiodic if there is some $y \in Y$ such that every open neighborhood $E \ni y$, we have $\P[1 \cdot^{\bar{\tau}} y \in E] > 0$.
\end{defn}
\begin{defn}
    A chain $Y$ is topologically irreducible if, for every point $y \in Y$ and open set $U \subseteq Y$, there is some $n = n(y, U) \in \N$ such that \[ \P[n \cdot y \in U] > 0. \]
\end{defn}
\begin{lem}\label{lem:top-irred}
    Suppose that, for every point $y \in Y$ and every open set $U \subseteq Y$, there is some sequence of nonnegative numbers $\tau_1, \dots, \tau_{2n} \geq 0$ such that $n \cdot^{\bar{\tau}} y \in U$. Then $Y$ is topologically irreducible.
\end{lem}
\begin{proof}
    Without loss of generality, we assume that $\tau_i \leq T$ for all $i$. Indeed, the action of $1$ corresponding to the sequence $\tau_1, \tau_2$ is the same as the action of $2$ corresponding to $\frac{\tau_1}{2}, 0, \frac{\tau_1}{2}, \tau_2$, so by making finitely many such substitutions we can ensure that $\tau_i \leq T$ for all $i$. We conclude by noting that the $\Z$-action is continuous and $\tau_1, \dots, \tau_{2n}$ lies in the support of our probability measure $\P_n$.
\end{proof}
\begin{defn}
    A function $V \colon Y \to \R_{\geq 0}$ satisfies the Lyapunov\textendash{}Foster drift condition if there exists some $0 < \alpha < 1$, $b > 0$, and a compact set $C \subseteq Y$ such that \[ \E[V(1 \cdot y)] \leq \alpha V(y) + b\mathds{1}_C \] for every $y \in Y$.
\end{defn}

The above conditions are exactly the hypotheses needed to verify uniform geometric ergodicity. The following is a version of the Perron\textendash{}Frobenius theorem for Markov chains, adapted to the continuous setting.
\begin{thm}[abstract Harris theorem, Theorem~2.3 in~\cite{BCZG}]\label{thm:abstract-harris}
    Let $Y$ be any of the chains above, and assume the following.
    \begin{enumerate}
        \item $Y$ admits an open small set.
        \item $Y$ is topologically irreducible.
        \item $Y$ is strongly aperiodic.
        \item There is a function $V \colon Y \to \R_{\geq 0}$ satisfying the Lyapunov\textendash{}Foster drift condition.
    \end{enumerate}
    Then $Y$ is $V$-uniformly geometrically ergodic.
\end{thm}

We spend the rest of this section verifying the hypotheses of the Theorem~\ref{thm:abstract-harris}. To verify uniform geometric ergodicity for the two-point chain, we first verify it for the one-point chain and the projective chain, which allows us to apply a result of Blumenthal\textendash{}Coti~Zelati\textendash{}Gvalani~\cite{BCZG}.

\begin{lem}\label{lem:irreducible-and-small}
    The one-point, projective, and two-point chains are all topologically irreducible and strongly aperiodic. Furthermore, each of these chains admits a small set.
\end{lem}
\begin{proof}
    Strong aperiodicity follows immediately from the fact that $\P[\tau_1 < \varepsilon] > 0$ for any $\varepsilon > 0$.
    
    Since the one-point chain is a special case of either the projective chain or the two-point chain, it suffices to prove topological irreducibility for those. We use Lemma~\ref{lem:top-irred} and consider separately
    \begin{enumerate}
        \item the projective chain:
            We say that a set $E \subseteq T^1 X$ is reachable from a point $(x, v)$, if there exist some $\tau_1, \tau_2, \dots, \tau_{2n} \geq 0$ such that $\hat{\Phi}_n^{\bar{\tau}}(x, v) \in E$. If $E, F \subseteq T^1 X$ are both sets, then we say that $E$ is reachable from $F$ if $E$ is reachable from every point in $F$.

            Define the sets $E_1 := (\{\pi/2\} \oplus \R) \oplus S^1$, $E_2 := (\{3\pi/2\} \oplus \R) \oplus (S^1 \setminus (\pm 1, 0))$, and $E_3 := \{((\pi/2, \pi/2), (0, 1))\}$. We want to show that $E_3$ is reachable from $T^1 X$, so we show the following three claims.
            \begin{enumerate}
                \item $E_1$ is reachable from $T^1 X$.

                    Fix $((x, y), v) \in T^1 X$. Since either $\sin(x) \neq 0$ or $\sin(y) \neq 0$, we see that, by taking $\tau_1 := 0$ and almost any $\tau_2 > 0$, we may assume without loss of generality that $y \neq 0$. By symmetry, assume that $\sin(y) > 0$. Then set $\tau_1 := \frac{5\pi/2 - x}{\sin(y)}$ and  $\tau_2 := 0$ to see the claim.

                \item $E_2$ is reachable from $E_1$.

                    Let $((x, y), v) \in E_1$. If $v = (\pm 1, 0)$, then set $\tau_1 := 0$, $\tau_2 := 9\pi/4 - y$, $\tau_3 := \pi \sqrt{2}$, $\tau_4 := 0$ to conclude. On the other hand, if $v_2 \neq 0$, then set $\tau_1 := 0$, $\tau_2 := 5\pi/2 - y$, $\tau_3 := \pi$, $\tau_4 := 0$ to conclude

                \item $E_3$ is reachable from $E_2$.

                    Let $((x, y), (v_1, v_2)) \in E_2$. Choose $\tau_3 > 0$ and $z \in [0, \pi]$ such that $\tau_3 \sin(z) \in 2\pi \Z + \pi$ and $\tau_3 \cos(z) = v_1/v_2$. Then set $\tau_1 := 0$, $\tau_2 := y - z + \pi$, and $\tau_4 = \pi/2 - z + 2\pi$ to conclude.
            \end{enumerate}
            We conclude by noting that any point $(y, w)$ is reachable from a point $(x, v)$ by first traveling to $E_3$ and then traveling to $(y, w)$ (using the same arguments as above, but in reverse).
        \item the two-point chain:
            Fix $(w, z) \in X^2 \setminus \Delta$ and $\varepsilon > 0$. As with the projective chain, we define the sets \[ E_1 := \{(x, y) \in X^2 \setminus \Delta \mid \sin(x_2) \text{ and } \sin(y_2) \text{ are $\Q$-linearly independent}\}, \] \[ E_2 := (B_{\varepsilon/3}(w_1) \oplus \R) \oplus (B_{\varepsilon/3}(z_1) \oplus \R), \] and \[ E_3 := \{(x, y) \in X^2 \setminus \Delta \mid \dist((x, y), E_2) \leq \varepsilon/3 \text{ and } \sin(x_1) \text{ and } \sin(y_1) \text{ are $\Q$-linearly independent}\}. \] We want to show that $B_{\varepsilon}(w, z)$ is reachable from $X^2 \setminus \Delta$, so we show the following 4 claims.
            \begin{enumerate}
                \item $E_1$ is reachable from $X^2 \setminus \Delta$.

                    Let $(x, y) \in X^2 \setminus \Delta$ and let $\delta > 0$. If $\sin(x_2) \neq \sin(y_2)$, then note that by varying $0 \leq \tau_1, \tau_2 \leq \delta$, the set of possible values of $(w_2, z_2)$ where $(w, z) = \tilde{\Phi}^{\bar{\tau}}_1(x, y)$ has nonempty interior, and hence (since $\Q$-linearly independent points are generic) $E_1$ is reachable from $(x, y)$. On the other hand, assume that $\sin(x_2) = \sin(y_2)$. If $\sin(x_1) \neq 0$ or $\sin(y_1) \neq 0$, then using $\tau_1 = 0$ and $0 \leq \tau_2 \leq \delta$ shows that a point $(x', y')$ with $\sin(x'_2) \neq \sin(y'_2)$ is reachable so we conclude. Otherwise, assume that $\sin(x_1) = \sin(y_1) = 0$. Then $\sin(x_2) \neq 0$, so the same argument with $0 \leq \tau_1 \leq \delta$ and $\tau_2 = 0$ shows that a point $(x', y')$ with $\sin(x'_1) \neq 0$, so we conclude.

                \item $E_2$ is reachable from $E_1$.

                    Since orbits of the map $\rot_{\alpha, \beta} \colon {(\R/\Z)}^2 \to {(\R/\Z)}^2$  given by $\rot_{\alpha, \beta}(x, y) := (x+\alpha, y+\beta)$ are dense when $\alpha, \beta$ are $\Q$-linearly independent, there is some satisfactory $\tau_1 > 0$ and $\tau_2 = 0$.
                \item $E_3$ is reachable from $E_2$.

                    We use the same procedure as in the first claim, noting that our construction moves the point a distance of at most $4\delta$. Choosing $\delta \leq \frac{\varepsilon}{12}$ allows us to conclude.

                \item $B_{\varepsilon}(w, z)$ is reachable from $E_3$.

                    We conclude by the same procedure as the second claim.
            \end{enumerate}
    \end{enumerate}

    Next, we use Lemma~\ref{lem:small-set-condition} check that each of the chains admits a small set.
    \begin{enumerate}
        \item the one-point chain:

            Set $y = (\frac{\pi}{2}, \frac{\pi}{2})$ and $n = 1$ with $\bar{\tau}^1 = (\pi, \pi)$. Then we compute that \[ D_{\bar{\tau}^1} \Phi^{\bar{\tau}^1}_1(y) = \begin{pmatrix} 1 & 0\\ 0 & -1\end{pmatrix}, \] which has rank $2$ and hence satisfies the hypotheses of Lemma~\ref{lem:small-set-condition}.
        \item the projective chain:
            
            Set $y = ((\frac{\pi}{2}, \frac{\pi}{2}), (0, 1))$ and $n = 2$ with $\bar{\tau}^2 = (\pi, \pi, \pi, \pi)$. Then we compute that \[ D_{\bar{\tau}^2} \hat{\Phi}^{\bar{\tau}^2}_2(y) = \begin{pmatrix} 1 & 0 & -1 & 0\\ 0 & -1 & 0 & 1\\ 0 & -\pi & 0 & 0\\ 0 & 0 & 0 & 0 \end{pmatrix}, \] which has rank $3$ as desired.
        \item the two-point chain:

            Set $y = ((0, \frac{\pi}{2}), (\frac{\pi}{2}, 0))$ and $n = 2$ with $\bar{\tau}^2 = (\pi, \pi, \pi, \pi)$. Then we compute that \[ D_{\bar{\tau}^2} \tilde{\Phi}^{\bar{\tau}^2}_2(y) = \begin{pmatrix} 1 & 0 & 1 & 0\\ 0 & 0 & \pi & 0\\ 0 & -\pi & 0 & 0\\ 0 & 1 & 0 & 1 \end{pmatrix}, \] which has rank $4$ as desired.
    \end{enumerate}
\end{proof}
\begin{lem}\label{lem:potential-exists}
    There exists an integrable function $V \colon X \to \R$ which satisfies the Lyapunov\textendash{}Foster drift condition.
\end{lem}
\begin{proof}
    For simplicity, we construct $V$ which satisfies the drift condition only locally at $(0, 0)$. Taking the maximum of four translated copies of such a $V$ together, one for each fixed point in $F$, will then satisfy the drift condition globally on $X$.

    We let $V(x, y) = {\max(x, y)}^{-\alpha}$ and $C = {[-r_0, r_0]}^2 \cap X$, where $\alpha, r_0 > 0$ will be chosen to be small (to start, we assume that $r_0 \leq \pi/2$ so $|\sin(t)| \geq |\frac{t}{2}|$ for all $t \in [-r_0, r_0]$). We want to show that, if $x, y \in {[-r_0, r_0]}^2$, then $\E[V(\Phi^{\bar{\tau}}_1(x, y))] < cV(x, y)$ for some $0 < c < 1$. Without loss of generality, assume that $y > 0$.

    If $x < 0$, then, as long as $2(T+1)r_0 < \pi$, we have $V(\Phi^{\bar{\tau}}_1(x, y)) \leq V(x, y)$. Besides, if $\tau_1, \tau_2 \geq 2$, we have $V(\Phi^{\bar{\tau}}_1(x, y)) \leq 2^{-\alpha}V(x, y)$. The event that $\tau_1, \tau_2 \geq 2$ has probability ${(T-2)}^2/T^2$, so we conclude that
    \begin{equation}\label{eq:bound1}
        \E[\Phi^{\bar{\tau}}_1(x, y)] \leq \left[(1-{(T-2)}^2/T^2) + 2^{-\alpha}\right]V(x, y)
    \end{equation}
    in the case where $x < 0$.

    On the other hand, assume that $x \geq 0$. There are two cases:
    \begin{enumerate}
        \item If $y \geq x$, then we split into three sub-cases:
            \begin{enumerate}
                \item Let $E$ be the event that $|x-\tau_1\sin(y)| \leq r_1 := \frac{y}{2T}$. Since $\sin(y) \geq \frac{y}{2}$, we see that $\P[E] \leq \frac{2}{T^2}$, and $V(\Phi^{\bar{\tau}}_1(x, y)) \leq 2^\alpha V(x, y)$ in the event $E$.
                \item On the other hand, let $F$ be the event that $r_1 < |x-\tau_1\sin(y)| \leq 2y$. Then $\P[F] \leq \frac{4}{T}$ and $V(\Phi^{\bar{\tau}}_1(x, y)) \leq r_1^{-\alpha} \leq {(2T)}^\alpha V(x, y)$ in the event $F$.
                \item Finally, in the complement of $E \cup F$, we see that $V(\Phi^{\bar{\tau}}_1(x, y)) \leq 2^{-\alpha}V(x, y)$.
            \end{enumerate}
            To conclude, we note that
            \begin{align}
                \E[V(\Phi^{\bar{\tau}}_1(x, y))] &= \E[V(\Phi^{\bar{\tau}}_1(x, y)) | E]\P[E] + \E[V(\Phi^{\bar{\tau}}_1(x, y)) | F]\P[F] + \E[V(\Phi^{\bar{\tau}}_1(x, y)) | {(E \cup F)}^c](1-\P[E \cup F]) \nonumber \\
                &\leq \left(\frac{2^{1+\alpha}}{T^2} + \frac{2^{2+\alpha}}{T^{1-\alpha}} + 2^{-\alpha}\right)V(x, y). \label{eq:bound2}
            \end{align}
        \item If $y < x$, then we split into four sub-cases:
            \begin{enumerate}
                \item As before, let $E$ be the event that $|x-\tau_1\sin(y)| \leq r_1 := \frac{y}{2T}$. Then $\P[E] \leq \frac{2r_1}{\frac12 y T} = \frac{2y}{T^2}$, and $V(\Phi^{\bar{\tau}}_1(x, y)) \leq {\left(\frac{y}{2}\right)}^{-\alpha}$. In order for $\P[E]$ to be positive, we need $x-yT \leq r_1$, and hence $y \geq \frac{x}{2T}$. Therefore, in the event $E$, we have $V(\Phi^{\bar{\tau}}_1(x, y)) \leq {(4T)}^\alpha V(x, y)$.
                \item Unlike before, we define $F$ to be the event that $r_1 < |x-\tau_1\sin(y)| \leq r_2 := \frac{x}{\sqrt{T}}$. Then $\P[F] \leq \frac{2r_2}{\frac12 y T} = \frac{4x}{yT^{3/2}}$. In order for $\P[F]$ to be positive, we need $x-yT \leq r_2$, and hence $x \leq y T$, so it follows that $\P[F] \leq \frac{4}{\sqrt{T}}$. In the event $F$, we have $V(\Phi^{\bar{\tau}}_1(x, y)) \leq r_1^{-\alpha} \leq {(2T^2)}^\alpha V(x, y)$.
                \item Finally, we define $G$ to be the event that $|x-\tau_1\sin(y)| > r_2$ and $|\tau_2\sin(r_2)| \leq 3x$. Then $\P[G] \leq \frac{3x}{\frac12 r_2 T} = \frac{6}{\sqrt{T}}$, and $V(\Phi^{\bar{\tau}}_1(x, y)) \leq r_2^{-\alpha} = T^{\alpha/2}V(x, y)$.
                \item In the complement of $E \cup F \cup G$, we have $V(\Phi^{\bar{\tau}}_1(x, y)) \leq 2^{-\alpha}V(x, y)$.
            \end{enumerate}
            To conclude, we note that
            \begin{align}
                \E[V(\Phi^{\bar{\tau}}_1(x, y))] &= \E[V(\Phi^{\bar{\tau}}_1(x, y)) | E]\P[E] + \E[V(\Phi^{\bar{\tau}}_1(x, y)) | F]\P[F] \nonumber \\
                &\quad + \E[V(\Phi^{\bar{\tau}}_1(x, y)) | G]\P[G] + \E[V(\Phi^{\bar{\tau}}_1(x, y)) | {(E \cup F \cup G)}^c](1-\P[E \cup F \cup G]) \nonumber \\
                &\leq \left(\frac{2^{1+2\alpha}}{T^{2-\alpha}} + \frac{2^{2+\alpha}}{T^{1/2 - 2\alpha}} + \frac{6}{T^{(1-\alpha)/2}} + 2^{-\alpha}\right)V(x, y). \label{eq:bound3}
            \end{align}
    \end{enumerate}

    Evaluating inequalities~\ref{eq:bound1},~\ref{eq:bound2}, and~\ref{eq:bound3} shows that, if we choose $\alpha = \frac{1}{20}$ and $T \geq 5 \cdot 10^{5}$, then $V$ satisfies the mixing condition with $c = 1-10^{-4}$.
\end{proof}

Together, Lemmas~\ref{lem:irreducible-and-small} and~\ref{lem:potential-exists} show that the assumptions of the Theorem~\ref{thm:abstract-harris} are satisfied, so the one-point chain is $V$-uniformly geometrically ergodic with the Lebesgue measure as the unique stationary measure. We note that the same $V$ (taken to be constant in the second coordinate of the projective chain) satisfies the drift condition for the projective chain, so the projective chain is also $V$-uniformly geometrically ergodic.

It remains to construct a function $W \colon X^2 \setminus \Delta \to \R_{> 0}$ for the two-point chain satisfying a Lyapunov\textendash{}Foster drift condition. A compact subset of $X^2 \setminus \Delta$ must avoid both the diagonal, $\Delta$, and fixed points in both coordinates, $F \times X \cup X \times F$. First, we appeal to a result of Blumenthal\textendash{}Coti~Zelati\textendash{}Gvalani~\cite{BCZG}, which handles the points near the diagonal. In the context of~\cite{BCZG}, the one-point chain is compact, so this function, $W$, satisfied the drift condition for the two-point chain. In our case, we will need to modify $W$ to deal with the lack of compactness of $X$ near fixed points.

\begin{defn}
    The top Lyapunov exponent is defined by \[ \lambda_1 := \lim_{n \to \infty} \frac{1}{n} \log |D_x \Phi^{\bar{\tau}}_n|. \]
\end{defn}

Given that $X$ has a stationary ergodic measure and the map $\Phi_1^{\bar{\tau}}$ is almost surely bounded in $C^1$ by a deterministic constant, it is well known (see, e.g. Kifer~\cite{Kifer}) that the limit which defines $\lambda_1$ exists and is almost surely a deterministic constant. The next result of Blumenthal\textendash{}Coti~Zelati\textendash{}Gvalani~\cite{BCZG} shows that verifying positivity of the top Lyapunov exponent is enough to handle points near the diagonal.

\begin{lem}[Proposition~4.5, Blumenthal\textendash{}Coti~Zelati\textendash{}Gvalani~\cite{BCZG}]\label{lem:almost-drift-condition}
    Suppose that $\lambda_1 > 0$ and that the one-point and projective chains are uniformly geometrically ergodic. Then there are $p,s > 0$, $0 < \gamma < 1$, and $\psi \colon X^2 \setminus \Delta \to \R_{\geq 0}$ continuous, bounded by constants $0 < c \leq \psi \leq C < \infty$, such that $W(x, y) := |x-y|^{-p}\psi(x, y)$ satisfies the inequality \[ \E[W(1 \cdot (x, y))] \leq \gamma W(x, y) \] for all $(x, y) \in X^2 \setminus \Delta$ with $\dist((x, y), \Delta) < s$.
\end{lem}

We have verified all the hypotheses for Lemma~\ref{lem:almost-drift-condition} except for positivity of the top Lyapunov exponent, which we verify now. To this end, we use another result of Blumenthal\textendash{}Coti~Zelati\textendash{}Gvalani~\cite{BCZG}, based on Furstenberg's criterion, which provides an easily-verifiable condition for the positivity of $\lambda_1$.

\begin{lem}[a special case of Blumenthal\textendash{}Coti~Zelati\textendash{}Gvalani~\cite{BCZG}, Proposition 3.3]\label{lem:lambda_1-condition}
    Assume that $X$ is uniformly geometrically ergodic and that there are $x \in X$ and $0 \leq \bar{\tau}_\star^n \leq T$ such that the map \[ \Psi_x \colon \bar{\tau}^n \mapsto \Phi^{\bar{\tau}^n}_n(x) \] is a submersion at $\bar{\tau}^n = \bar{\tau}^n_\star$. Also assume that the restriction of $D_{\bar{\tau}^n} D_x \Psi_x$ to $\ker D_{\bar{\tau}^n} \Psi_x$ is surjective with range $T_{D_x \Psi_x(\bar{\tau}^n_\star)} \SL_2(\R)$ at $\bar{\tau}^n = \bar{\tau}^n_\star$.

    Then $\lambda_1 > 0$.
\end{lem}
\begin{proof}
    In Blumenthal\textendash{}Coti~Zelati\textendash{}Gvalani~\cite{BCZG}, Lemma~\ref{lem:lambda_1-condition} is proven in a more general setting where $X$ is a compact Riemannian manifold. In our setting, $X$ is not compact. However, inspecting the proof, the exact same arguments go through by assuming only $\sigma$-compactness.
\end{proof}

We now check the last condition for exponential mixing.
\begin{lem}
    The top Lyapunov exponent for $X$, $\lambda_1$, is positive.
\end{lem}
\begin{proof}
    We verify the hypotheses of Lemma~\ref{lem:lambda_1-condition}. Let $x := (\frac{\pi}{2}, \frac{\pi}{2})$ and $n := 3$ with $\bar{\tau}^n_\star := (\frac{\pi}{2}, \pi, \pi, \pi, \pi/2, \pi/2)$. Then we compute \[ D_{\bar{\tau}^n} \Psi_x(\bar{\tau}^n_\star) = \begin{pmatrix} 1 & 0 & 1 & 0 & 1 & 0\\ 0 & 0 & \pi & 0 & 0 & 1\end{pmatrix} \] and \[ D_{\bar{\tau}^n} D_x \Psi_x(\bar{\tau}^n_\star) = \begin{pmatrix} -\pi^3 & 0 & 0 & 0 & 0 & 0\\ \pi^2 & 0 & -\frac{\pi^2}{2} & 0 & 0 & 0\\ -\frac{\pi}{2} - \pi^4 & -1 & -\frac{\pi}{2} & 1 & -\frac{\pi}{2} & 0\\ \pi^3 & 0 & 0 & 0 & 0 & 0\end{pmatrix}, \] where we make the identification \[ \begin{pmatrix} a & b\\ c & d \end{pmatrix} \mapsto \begin{pmatrix} a\\ b\\ c\\ d\end{pmatrix} \] for the range of $D_x \Psi_x$. We note that $D_{\bar{\tau}^n} \Psi_x(\bar{\tau}^n_\star)$ has full rank as desired, and has kernel spanned by the columns of the matrix \[ K := \begin{pmatrix} 1 & 0 & 0 & 0\\ 0 & 1 & 0 & 0\\ 0 & 0 & 1 & 0\\ 0 & 0 & 0 & 1\\ -1 & 0 & -1 & 0\\ 0 & 0 & -\pi & 0\end{pmatrix}. \] Finally, we note that \[ (D_{\bar{\tau}^n} D_x \Psi_x(\bar{\tau}^n_\star))K = \begin{pmatrix} -\pi^3 & 0 & 0 & 0\\ \pi^2 & 0 & \frac{-\pi^2}{2} & 0\\ -\pi^4 & -1 & 0 & 1\\ \pi^3 & 0 & 0 & 0\end{pmatrix}, \] which has rank $3 = \dim \SL_2(\R)$ as desired.
\end{proof}

From Lemma~\ref{lem:almost-drift-condition}, it follows that $W$ satisfies the inequality for the drift condition near the diagonal. Now, we modify $W$ to account for the fixed points.
\begin{prop}
    For sufficiently small $\alpha > 0$, the function $W'(x, y) := W(x, y) + \alpha(V(x) + V(y))$ satisfies the Lyapunov\textendash{}Foster drift condition for the two-point chain, where $W$ is given by Lemma~\ref{lem:almost-drift-condition} and $V$ is given by Lemma~\ref{lem:potential-exists}.
\end{prop}
\begin{proof}
    Let $p, s, \varphi, \gamma$ be given by Lemma~\ref{lem:almost-drift-condition}. Without relabeling, let $\gamma$ be large enough to satisfy the drift condition inequality for $V$ (so $\gamma \geq 1-10^{-4}$, as in Lemma~\ref{lem:potential-exists}). Let $r_0$ be as in the proof of Lemma~\ref{lem:potential-exists} and define $C := \{(x, y) \in X^2 \mid |x-y| \geq s \text{ and } \dist(\{x, y\}, F) \geq \varepsilon\}$ where $F := {\{0, \pi\}}^2$ is the set of fixed points and $\varepsilon > 0$ will be chosen later in the proof. To start, let $\eta \in (0, r_0)$ be chosen small enough so that $\dist(1 \cdot^{\bar{\tau}} x, F) < \eta$ implies $\dist(x, F) < \frac{s}{2}$ for any $0 \leq \bar{\tau} \leq T$, and let $\varepsilon \leq \eta$. We need to show that, if $(x, y) \not\in C$, then \[ \E[W'(1 \cdot (x, y))] \leq \gamma'W'(x, y), \] where $0 < \gamma' < 1$. Without loss of generality, there are three cases to consider.
    \begin{enumerate}
        \item If $\dist(x, F), \dist(y, F) < \eta$, then $|x-y| < s$ and therefore \[ \E[W(1 \cdot (x, y)) + \alpha(V(1 \cdot x) + V(1 \cdot y))] \leq \gamma\left(W(x, y) + \alpha(V(x) + V(y))\right) \] by linearity of expectation.
        \item If $\dist(x, F) < \eta$ but $\dist(y, F) \geq \eta$, then there are two sub-cases to consider.
            \begin{enumerate}
                \item If $|x-y| < s$, then by choosing $\beta := \min \{\dist(1 \cdot z, F) \mid 0 \leq \bar{\tau} \leq T \text{ and } \dist(z, F) \geq \eta \}$, $\gamma' := (\gamma+1)/2$ and $\alpha := (\gamma' - \gamma)\beta^{1/20}s^{-p}\min \psi$ we compute
                    \begin{align*}
                        \E[W(1 \cdot (x, y)) + \alpha(V(1 \cdot x) + V(1 \cdot y))] &\leq \gamma W(x, y) + \gamma \alpha V(x) + \alpha \E[V(1 \cdot y)]\\
                        &\leq \gamma W(x, y) + \gamma \alpha V(x) + \frac{\alpha}{2} \beta^{-1/20}\\
                        &\leq \gamma' W(x, y) + \gamma' \alpha V(x) + \gamma' \alpha V(y),
                    \end{align*}
                    as desired.
                \item Otherwise, $\dist(x, F) < \varepsilon$. Let $\rho := \min \{|(1 \cdot x)-(1 \cdot y)| \mid 0 \leq \bar{\tau} \leq T \text{ and } |x-y| \geq s\}$ and $\beta$ as in the previous subcase, choose \[ \varepsilon := {\left(\frac{(\gamma'-\gamma)\alpha}{\rho^{-p} \max \psi + \beta^{-1/20}}\right)}^{20}, \] and compute
                    \begin{align*}
                        \E[W(1 \cdot (x, y)) + \alpha(V(1 \cdot x) + V(1 \cdot y))] &\leq \E[W(1 \cdot (x, y))] + \gamma \alpha V(x) + \alpha \E[V(1 \cdot y)]\\
                        &\leq \E[W(1 \cdot (x, y))] + \gamma \alpha V(x) + \alpha \E[V(1 \cdot y)]\\
                        &\leq \gamma' W(x, y) + \gamma' \alpha V(x) + \gamma' \alpha V(y),
                    \end{align*}
                    as desired.
            \end{enumerate}
        \item If $\dist(x, F), \dist(y, F) \geq \eta$, then $|x-y| < s$ (since we assume $(x, y) \not\in C$) and so we handle this case the same as subcase (2a) (but for both $x$ and $y$ instead of only $y$).
    \end{enumerate}
\end{proof}

We conclude from Theorem~\ref{thm:abstract-harris} that the two-point chain is uniformly geometrically ergodic. By Proposition~\ref{prop:uge-implies-mixing}, the maps ${\{\Phi^{\bar{\tau}}_n\}}_n$ mix exponentially in Bressan's sense for almost every choice of $\bar{\tau}$.

\section{Subexponential mixing with large deterministic durations}
\begin{figure}[htb]
    \centering
    \includegraphics{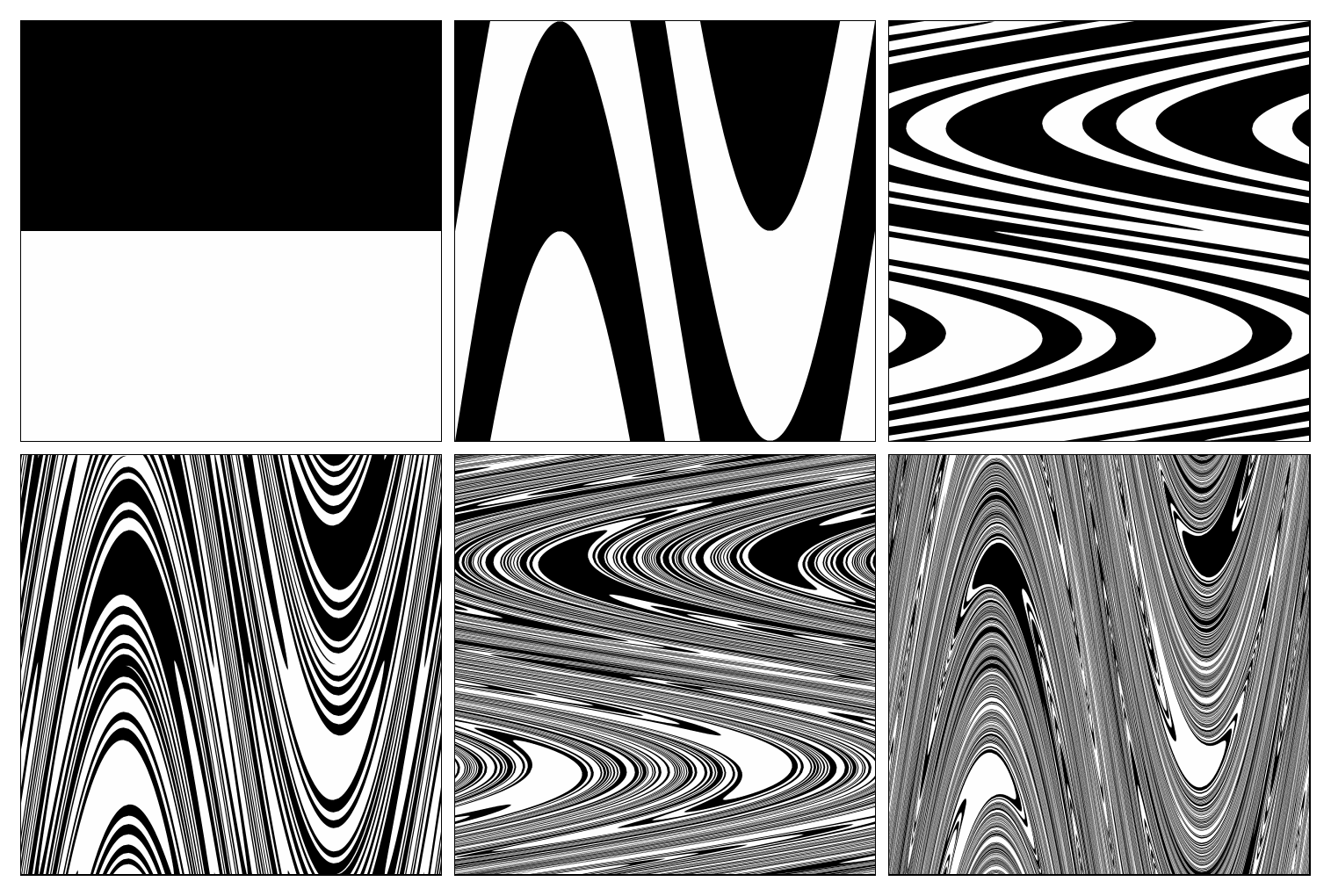}
    \caption{The flow for $T = 2\pi$ at $n=1,2,3$ and intermediate stages. The flow mixes exponentially away from the four clumps at the extrema of the sine fields, but the clump sizes shrink algebraically.}
\end{figure}

Last, we show that simply running each sine field for a sufficiently long duration is not sufficient to ensure mixing. Indeed, taking the notation of the previous section, let $T \in 2\pi\N$ and let $\tau_n = T$ for all $n$. Then the point $p := (\frac{\pi}{2}, \frac{\pi}{2})$ is a fixed point of $\Phi_1^{\bar{\tau}}$, and $D_x \Phi_1^{\bar{\tau}}(p) = 0$. Let $0 < h < 1 < C$ be such that \[ |\Phi_1^{\bar{\tau}}(x)-p| \leq |x-p| + C|x-p|^2 \quad \text{for all $x \in B_h(p)$.} \] Since $\bar{\tau}$ is constant, we have \[ \Phi_n^{\bar{\tau}} = \underbrace{\Phi_1^{\bar{\tau}} \circ \Phi_1^{\bar{\tau}} \circ \dots \circ \Phi_1^{\bar{\tau}}}_{n\text{ times}}. \] It follows that, if $x \in B_{\frac{h}{enC}}(p)$, then $\Phi_n^{\bar{\tau}}(x) \in B_h(p)$. By symmetry, the same holds for ${(\Phi_1^{\bar{\tau}})}^{-1}$, so plugging $B = B_{\frac{h}{enC}}(p)$ into the definition of the mixing scale shows that \[ \mix(b^{\bar{\tau}}_\alpha) \geq c\alpha^{-1} \geq c\|D_x b^{\bar{\tau}}_\alpha\|_{L^1}^{-1}. \]

\bibliographystyle{plain}
\bibliography{mixing}
\end{document}